\documentclass[reqno]{amsart}

\usepackage{mathrsfs}
\usepackage{amsthm}
\usepackage{amscd}
\usepackage{amsmath}
\usepackage{amssymb}
\usepackage{extarrows}
\usepackage[colorlinks]{hyperref}

\theoremstyle{definition}
\newtheorem{defn}{Definition}[section]

\theoremstyle{plain}
\newtheorem{thm}{Theorem}[section]
\newtheorem{prop}[thm]{Proposition}
\newtheorem{lem}[thm]{Lemma}

\newtheorem*{clm}{Claim}

\theoremstyle{remark}
\newtheorem*{rmk}{Remark}

\newcommand{\ve}{\varepsilon}

\newcommand{\hrho}{\hat{\rho}}

\newcommand{\mbba}{\mathbb{A}}

\newcommand{\mbbn}{\mathbb{N}}

\newcommand{\mbbp}{\mathbb{P}}
\newcommand{\mbbq}{\mathbb{Q}}
\newcommand{\mbbr}{\mathbb{R}}

\newcommand{\mbbz}{\mathbb{Z}}

\newcommand{\mbfa}{\mathbf{A}}
\newcommand{\mbfb}{\mathbf{B}}

\newcommand{\mbfg}{\mathbf{G}}
\newcommand{\mbfh}{\mathbf{H}}

\newcommand{\mbfp}{\mathbf{P}}

\newcommand{\mclh}{\mathcal{H}}

\newcommand{\mclm}{\mathcal{M}}

\newcommand{\mclx}{\mathcal{X}}

\newcommand{\tx}{\tilde{x}}

\newcommand{\tX}{\tilde{X}}
\newcommand{\tY}{\tilde{Y}}

\newcommand{\dif}{\,\mathrm{d}}

\newcommand{\GL}{\mathrm{GL}}
\newcommand{\SL}{\mathrm{SL}}

\DeclareMathOperator{\tr}{tr}
\DeclareMathOperator{\supp}{supp}

\title[Lyapunov maximizing measures for balanced pairs]{Lyapunov maximizing measures for balanced pairs of matrices}

\author[R. Gao]{Rui Gao}

\address{Rui Gao: School of Mathematical Sciences, Qufu Normal University, Jining 273165, China}
\email{gaoruimath@qfnu.edu.cn}

\begin{document}

\begin{abstract}
We show that every balanced pair (see Definition~\ref{defn:balanced pair}) of real $2\times 2$ matrices admits a unique Lyapunov maximizing measure, and the measure is always Sturmian. 
\end{abstract}

\maketitle


\section{Introduction}

Let $\GL(2,\mbbr)$ denote the general linear group of $2\times 2$ matrices over $\mbbr$. We shall mainly consider matrices in $\GL(2,\mbbr)$ with positive determinant, so let us denote
\[
\mbfg:=\{X\in \GL(2,\mbbr) : \det X >0\}.
\]
For each $X\in\mbfg$, let $\rho(X)$ denote its spectral radius. Given $\mbba=(A,B)\in\mbfg^2$, let
\[
\hrho(\mbba):=\limsup_{n\to\infty} \left(\max\{\rho(X_1\cdots X_n):X_1,\cdots,X_n\in \{A,B\} \}\right)^{1/n} \in (0,+\infty),
\]
called the {\bf joint spectral radius} of $\mbba$. If the limit superior above can be attained by a finite product $X_1\cdots X_n$,  then $X_1\cdots X_n$ is called a {\bf spectrum maximizing product} of $\mbba$. The {\bf finiteness conjecture} of Lagarias and Wang \cite{LW95} asserts that a spectrum maximizing product always exists for any $\mbba$. Actually the concept of joint spectral radius and the finiteness conjecture can be stated under more general settings (see, for example, \cite{LW95,Jun09}), but in this note we are only concerned about the special version described above. 

Let us restate the finiteness conjecture in the language of ergodic theory. Let $\mbbn=\{1,2,\cdots\}$ and let $\sigma:\{0,1\}^\mbbn\to\{0,1\}^\mbbn$ be the one-sided full shift. Let $\mclm$ denote the collection of $\sigma$-invariant Borel probability measures on $\{0,1\}^\mbbn$, endowed with the weak-* topology. Given $\mbba=(A_0,A_1)\in \mbfg^2$ and a matrix norm $\|\cdot\|$, for every $\mu\in\mclm$, by the sub-additive ergodic theorem, the limit
\begin{equation}\label{eq:LyaExp}
  \chi(\mbba,\mu):=\lim_{n\to\infty}\frac{1}{n}\int_{\{0,1\}^\mbbn} \log \|A_{i_n}\cdots A_{i_2}A_{i_1}\|\dif\mu(i_1i_2\cdots i_n\cdots)
\end{equation}
exists (independent of the choice of $\|\cdot\|$), and it is called the (top) {\bf Lyapunov exponent} of $\mbba$ with respect to $\mu$. It is well-known that
\[
\sup\{\chi(\mbba,\mu):\mu\in \mclm\} =\log \hrho(\mbba).
\]
If $\mu\in\mclm$ attains the supremum above, then it is called a ({\bf Lyapunov}) {\bf maximizing measure} of $\mbba$. Then the finiteness conjecture can be restated as: there always exists a maximizing measure supported on a periodic orbit of $\sigma$. The maximizing measures of $\mbba$ introduced above can also be defined under more general settings; see, for example, \cite{Mor13,Boc18}.   

The finiteness conjecture was disproved by Bousch and Mairesse \cite{BM02}, who constructed the first examples of $\mbba\in\mbfg^2$ which admits no spectrum maximizing product; such $\mbba$ with be called a {\bf finiteness counter-example}. The basic idea of their construction is roughly as follows: choose a pair $(A,B)\in \mbfg^2$ carefully and show that there exist (uncountably many) finiteness counter-examples in the one-parameter family $(A,tB),t>0$. To the best of our knowledge, all the known finiteness counter-examples in the literature were constructed in this way, including \cite{BM02,BTV03,Koz05,HMST11,MS13,JP18,Ore18} among others (we shall provide a different approach in Proposition~\ref{prop:family}). Moreover, it turns out that every pair $(A,B)\in \mbfg^2$ in all the constructions cited above is a balanced pair defined below; see Proposition~\ref{prop:balanced} for more details. 

Let $I\in\mbfg$ denote the identity matrix and let $\SL(2,\mbbr)=\{X\in\mbfg:\det X=1\}$. Define
\[
\mbfa:=\{A\in \SL(2,\mbbr) : \tr A\ge 2\}\setminus\{I\}=\mbfh\cup\mbfp,
\]
where $\mbfh=\{A\in \mbfa: \tr A>2\}$ and $\mbfp=\{A\in \mbfa: \tr A=2\}$. Define
\begin{equation}\label{eq:hyperbolic}
  H_{s,u}^\lambda:=
\begin{pmatrix}
  s & u \\
  1 & 1 
\end{pmatrix}
\begin{pmatrix}
  \lambda & 0 \\
  0 & \lambda^{-1}
\end{pmatrix}
\begin{pmatrix}
  s & u \\
  1 & 1 
\end{pmatrix}^{-1}\in\mbfh \qquad\text{for}\quad \lambda>1, u<0<s,
\end{equation}
and
\begin{equation}\label{eq:parabolic}
  P_x:= \begin{pmatrix} 1 & 0 \\ x & 1 \end{pmatrix}\in\mbfp \qquad\text{for}\quad x>0.
\end{equation}
Let $A^t$ denote the transpose of a matrix $A$. 

\begin{defn}\label{defn:balanced pair}
$\{A,B\}\subset \mbfa$ is called a {\bf balanced pair}, if there exists $T\in \GL(2,\mbbr)$ such that for $\mbba=\{T^{-1}AT, T^{-1}BT\}$ (called a {\bf normal form} of $\{A,B\}$), one of the following mutually exclusive conditions holds: 
\begin{itemize}
  \item [(h)] $\mbba=\{H_{s_1,u_1}^{\lambda_1}, H_{s_2,u_2}^{\lambda_2}\}$ with $u_2<u_1<0<s_1<s_2$;
  \item [(m)] $\mbba=\{H_{s,u}^\lambda, P_x\}$;
  \item [(p)] $\mbba=\{P_x,P_y^t\}$.
\end{itemize}
In the case (h), $\{A,B\}$ is called a {\bf co-parallel pair}. For $\{A,B\}\subset \mbfg$, if $\{aA,bB\}\subset\mbfa$ is a balanced (resp. co-parallel) pair for some $a,b\in\mbbr$, then $\{A,B\}$ is also called a {\bf balanced} (resp. {\bf co-parallel}) {\bf pair}. A balanced pair $\{A,B\}\subset \mbfg$ will also be denoted as an ordered pair $(A,B)\in \mbfg^2$.  
\end{defn} 
\begin{rmk}
 It is easy to see that our definition of co-parallel pairs coincides with that given at the beginning of \cite{JS90}, or that in \cite[Definition~3.3]{Las25}. Taking the finiteness counter-examples given by \cite{BTV03,HMST11} into consideration, we generalize this definition to balanced pairs, which is close to the concept ``well oriented" in \cite[Definition~2.1]{PS21}. 
\end{rmk}

To provide more details of the argument originated from \cite{BM02} and to introduce our main result, let us recall the concept of Sturmian measure first. For each $\alpha\in[0,1]$, let $s_\alpha\in\{0,1\}^\mbbn$ be defined by 
\begin{equation}\label{eq:rotation}
  s_\alpha=i_1i_2\cdots\cdots, \qquad i_n=\lfloor n\alpha\rfloor-\lfloor (n-1)\alpha\rfloor, \forall n\in \mbbn.
\end{equation}
Here for $x\in \mbbr$, $\lfloor x\rfloor$ denotes the maximal integer not exceeding $x$. Let $S_\alpha$ denote the $\omega$-limit set of $s_\alpha$. It is well-known that the sub-shift $S_\alpha$ supports a unique measure in $\mclm$, called the {\bf Sturmian measure of slope $\alpha$}, and denoted by $\mu_\alpha$. Note that $\mu_\alpha$ is supported on a periodic orbit iff its slope $\alpha$ is rational. The statements above can be found in, for example, \cite[Theorem~2.1]{MS13}. 
For certain (balanced) pair $(A,B)\in\mbfg^2$ and an associated interval $I\subset(0,+\infty)$, in \cite{BM02,Koz05,HMST11,MS13,JP18} among others, the authors studied the one-parameter family $(A,tB),t\in I$ and proved the following:
\begin{itemize}
  \item for every $t\in I$, $(A,tB)$ admits a unique maximizing measure, which is a Sturmian measure of slope $\alpha(t)$;
  \item the map $t\mapsto \alpha(t),t\in I$ is continuous and non-constant (even of range $[0,1]$).
\end{itemize}
As a result, for every $t\in I$ with $\alpha(t)\notin\mbbq$, $(A,tB)$ serves as a finiteness counter-example. 

A result of J{\o}rgensen and Smith \cite[Theorem~5-4]{JS90} (see also Theorem~\ref{thm:J-S} for a slightly generalized version) is closely related to the way of searching finiteness counter-examples in co-parallel pairs described above. It seems that this fact did not receive much attention until the recent works \cite{PS21,Las25}. In particular, based on \cite[Theorem~5-4]{JS90}, recently Laskawiec \cite[Theorem~9.8]{Las25} proved that for any co-parallel pair $\mbba\in \mbfg^2$, $\mbba$ admits exactly one Sturmian maximizing measure. Using a variation of J{\o}rgensen and Smith's theorem (see Proposition~\ref{prop:J-S}) as a key ingredient, we can prove that the maximizing measure of such $\mbba$ is always unique, confirming a suspicion of Laskawiec \cite[Note~9.9]{Las25}. Our main result in this note reads as follows.

\begin{thm}\label{thm:main}
  Let $\mbba\in \mbfg^2$ be a balanced pair. Then $\mbba$ admits a unique Lyapunov maximizing measure; moreover, this measure is Sturmian. 
\end{thm}

The rest of this note is organized as follows. In \S~\ref{sse:words} we recall the notion of balanced words, which is closely related to Sturmian measures.  In \S~\ref{sse:pair} we discuss basic properties of balanced pairs.  In \S~\ref{se:J-S} we prove Proposition~\ref{prop:J-S}, following the original argument in \cite{JS90}. In \S~\ref{se:main} we deduce Theorem~\ref{thm:main} from Proposition~\ref{prop:J-S}. Some further discussions related to Theorem~\ref{thm:main} will be given in \S~\ref{se:FD}. Finally, in \S~\ref{se:J-S proof} we provide a proof of Theorem~\ref{thm:J-S}.

\section{Preliminaries}\label{se:pre}

\subsection{Balanced words}\label{sse:words}

Given an integer $n\ge 0$, an element in $\{0,1\}^n$ is called a {\bf finite word} of {\bf length} $n$. By definition, $\{0,1\}^0$ consists of a unique element, called the {\bf empty word}, denoted by $\epsilon$.  A finite word of length $n\ge 1$ is denoted by $i_1\cdots i_n$ for $i_1,\cdots,i_n\in\{0,1\}$. Let $\{0,1\}^*=\cup_{n=0}^\infty\{0,1\}^n$ denote the collection of all finite words. Recall that $\{0,1\}^*$ can be seen as a free monoid with generator $0,1$ and identity $\epsilon$, by defining the product of $u=i_1\cdots i_m$ and $v=j_1\cdots j_n$ as $uv=i_1\cdots i_m j_1\cdots j_n$. 
For each $n\ge 0$, let $u^n$ denote the $n$-th power of a finite word $u$; in particular, $u^0=\epsilon$. An element in $\{0,1\}^\mbbn$ is called an {\bf infinite word}, which is denoted by $i_1i_2\cdots$ for $i_1,i_2,\cdots\in\{0,1\}$. 
The product $uv$ of a finite word $u$ and an infinite word $v$ can be defined in a similar way. 

From now on, by saying a {\bf word}, we mean either a finite or an infinite word. Given a word $x$ and a finite word $v$, $v$ is called a {\bf factor} of $x$, if there exists a finite word $u$ and a word $w$ such that $x=uvw$. For a finite word $w$, let $|w|$ denote its length. When $w=i_1\cdots i_n$ for some $n\ge 1$, define $|w|_1=i_1+\cdots+i_n$. Define $|\epsilon|_1=0$. A word $w$ is called {\bf balanced}, if for any two factors $u$ and $v$ of $w$, $|u|=|v|$ implies that $||u|_1-|v|_1|\le 1$.
 

Sturmian measures mentioned in the introduction are closely related to balanced words. In the three lemmas below we collect some well-known facts about balanced words that will be used later. For more information on balanced words and their relation to Sturmian measures, we refer to \cite{BM02,Lot02,MS13} and references therein. Recall that for each $\alpha\in [0,1]$, $S_\alpha$ denotes the $\omega$-limit set of $s_\alpha$ defined by \eqref{eq:rotation}. An infinite word $s$ is called {\bf recurrent}, if any factor of $s$ occur infinitely often in $s$; or equivalently, $s$ is contained in the $\omega$-limit set of itself.

\begin{lem}\label{lem:recurrent and balanced}
  For $s\in\{0,1\}^\mbbn$, $s$ is recurrent and balanced iff there exists $\alpha\in[0,1]$ such that $s\in S_\alpha$.
\end{lem}
\begin{proof}
  See, for example, \cite[Theorem~2.1]{MS13}.
\end{proof}
\begin{lem}\label{lem:Sturmian slope}
Given $\alpha\in [0,1]$, the following hold.
\begin{itemize}
  \item [(i)] The sub-shift $S_\alpha$ is uniquely ergodic.
  \item [(ii)] $|\sum_{m\le k<m+n} i_k-n\alpha|<1$ holds for any $i_1i_2\cdots\in S_\alpha$ and any $m,n\ge 1$.
\end{itemize}
\end{lem}
\begin{proof}
 For (i), see, for example, \cite[Theorem~2.1]{MS13}. For (ii), see, for example, \cite[Corollary~1.4]{BM02}.
\end{proof}

\begin{lem}\label{lem:not balanced}
  Let $x$ be a word which is not balanced. Then there exists a word $w$ such that both $0w0$ and $1w1$ are factors of $x$.
\end{lem}
\begin{proof}
  See, for example, \cite[Proposition~2.1.3]{Lot02}.
\end{proof}

\subsection{Balanced pairs of matrices}\label{sse:pair}
In this subsection we summarize properties of balanced pairs of matrices that will be used later, listed in Lemma~\ref{lem:cone} and Lemma~\ref{lem:balanced pair} below. Most of them can be founded in \cite{JS90,PS21,Las25}.

Let $A=\begin{pmatrix}a & b \\ c & d \end{pmatrix}\in \mbfg$. If $a,b,c,d\ge 0$, then $A$ is called {\bf non-negative}; if $a,b,c,d>0$, then $A$ is called {\bf positive}. We shall identify the projective space  $\mbbp^1$ of $\mbbr^2$ with $\mbbr\cup\{\infty\}$, via the projective map $(x,y)^t\mapsto x/y,\, (x,y)^t\in\mbbr^2\setminus\{(0,0)^t\}$. Then $A$ naturally induces a real M\"{o}bius transformation 
\[
\underline{A}:\mbbp^1\to\mbbp^1, \quad x\mapsto \frac{ax+b}{cx+d}.
\] 
When $A\in \mbfa$, the following facts are evident.
\begin{itemize}
\item $A\in\mbfh$ iff $\underline{A}:\mbbp^1\to \mbbp^1$ has exactly two (hyperbolic) fixed points, one attracting and the other repelling, denoted by $s_A$ and $u_A$ respectively. In this case, by Perron-Frobenius theorem, the following three conditions are equivalent:
 \begin{itemize}
   \item $A$ is positive;
   \item $s_A,u_A\in\mbbr$ and $u_A<0<s_A$;
   \item $A=H_{s,u}^\lambda$ for some $\lambda>1$ and $u<0<s$. 
 \end{itemize}
Moreover, if $A=H_{s,u}^\lambda$, then $s=s_A$ and $u=u_A$.
  \item $A\in\mbfp$ iff $\underline{A}:\mbbp^1\to \mbbp^1$ has a unique (parabolic) fixed point, denoted by $p_A$. In this case, $A$ is non-negative iff one of the following holds:
\begin{itemize}
  \item either $p_A=0$ and $A=P_x$ for some $x>0$;
  \item or $p_A=\infty$ and $A=P_x^t$ for some $x>0$.
\end{itemize}
\end{itemize}

Given $\mbba=(A_0,A_1)\in\mbfg^2$, define $[\epsilon]_{\mbba}=I$, and define $[w]_{\mbba}=A_{i_1}\cdots A_{i_n}$ for each $w=i_1\cdots i_n\in\{0,1\}^*\setminus\{\epsilon\}$. Then $[\cdot]_\mbba$ defines a homomorphism from $\{0,1\}^*$ to $\mbfg$. For an open subset $U$ of $\mbbp^1$, let $\overline{U}$ denote its closure in $\mbbp^1$. The following facts are evident.

\begin{lem}\label{lem:cone}
  Let $\mbba=(A_0,A_1)\in\mbfa^2$ be one of the norm forms given in Definition~\ref{defn:balanced pair}. Define
\begin{equation}\label{eq:cone}
I_{\mbba}^+  =\left\{\begin{array}{ccl}
  (s_1,s_2) &,&  \text{case (h)} \\
  (0,s) &,&  \text{case (m)} \\
  (0,+\infty) &,&  \text{case (p)} 
\end{array}\right.,\quad
I_{\mbba}^- =\left\{\begin{array}{ccl}
  (u_2,u_1) &,&  \text{case (h)} \\
  (u,0) &,&  \text{case (m)} \\
  (-\infty,0) &,&  \text{case (p)} 
\end{array}\right..
\end{equation}
Given a non-empty finite word $x$, denote $X=[x]_{\mbba}$. Then the following hold.
\begin{itemize}
  \item [(i)] $X\in\mbfa$, $\underline{X}(I_{\mbba}^+)\subset I_{\mbba}^+$ and $\underline{X}^{-1}(I_{\mbba}^-)\subset I_{\mbba}^-$.
  \item [(ii)] $X\in\mbfp$ iff there exists $i\in\{0,1\}$ such that $x=i^{|x|}$ and $A_i\in\mbfp$. 
  \item [(iii)] If $X\in\mbfh$, then $s_X\in \overline{I_{\mbba}^+}$ and  $u_X\in \overline{I_{\mbba}^-}$.
  \item [(iv)] If $x\ne 0^{|x|},1^{|x|}$, then $\underline{X}(\overline{I_{\mbba}^+})\subset I_{\mbba}^+$ and $\underline{X}^{-1}(\overline{I_{\mbba}^-})\subset I_{\mbba}^-$. 
\end{itemize}
\end{lem}

Lemma~\ref{lem:balanced pair}~(iv) below is a special case of \cite[Lemma~5-3]{JS90} when $\mbba$ is a co-parallel pair. To prove that this statement holds for all balanced pairs, let us introduce the notion of crossing pairs for preparation. For $\{A,B\}\subset \mbfh$, if there exists $T\in\GL(2,\mbbr)$ such that $\{T^{-1}AT,T^{-1}BT\}=\{H_{s_1,u_1}^{\lambda_1}, H_{s_2,u_2}^{\lambda_2}\}$ with $u_1<u_2<0<s_1<s_2$, then $\{A,B\}$ is called a {\bf crossing pair}. It is easy to see that this definition coincides with \cite[Definition~3.1]{Las25} for pairs of matrices in $\mbfh$.
\begin{lem}\label{lem:balanced pair}
  Let $\mbba=(A,B)\in\mbfa^2$ be a balanced pair. 
\begin{itemize}
  \item [(i)] $[w]_\mbba\in\mbfa$ holds for any non-empty finite word $w$; if $w\ne 0^{|w|},1^{|w|}$ additionally, then $[w]_\mbba\in\mbfh$. 
  \item [(ii)] For each $X\in \{A,B\}$ and each $Y\in \{AB,BA\}$, $\{X,Y\}$ is a balanced pair. 
  \item [(iii)] $\{AB,BA\}$ is a crossing pair.
  \item [(iv)] $\tr ((AB)^2) >\tr (A^2B^2)$.
\end{itemize}
\end{lem}

\begin{proof}
We may assume that $\mbba$ is in a normal form. Then assertions (i)(ii) follows from Lemma~\ref{lem:cone}. To prove (iii)(iv), denote $C=AB$ and $D=BA$ below. We may further assume that the right end point of $I_\mbba^+$ is a fixed point of $\underline{B}$, so $\underline{B}(x)>x$ holds for any $x\in I_\mbba^+\cup I_\mbba^-$.  

(iii). By Lemma~\ref{lem:cone}, $C,D\in\mbfh$, $s_C,s_D\in I_{\mbba}^+$ and
\[
\underline{D}(\underline{B}(s_C))=\underline{B}(\underline{C}(s_C))=\underline{B}(s_C)\in I_{\mbba}^+,
\]
so $s_D=\underline{B}(s_C)>s_C$. A similar argument shows that $u_D=\underline{B}(u_C)>u_C$. Therefore $\{C,D\}$ is a crossing pair.

(iv). Since $\{C,D\}$ is a crossing pair, it is known that $\rho(C)\rho(D)>\rho(CD)$. See, for example, \cite[Theorem~2.5]{PS21}, or the displayed line in the last paragraph of \cite[Proof of Corollary~5.4]{Las25}. Then it follows that 
\[
\rho((AB)^2)= \rho(C)\rho(D) > \rho(CD) = \rho(A^2B^2),
\]
which is equivalent to $\tr ((AB)^2) >\tr (A^2B^2)$.
\end{proof}

\section{A theorem of J{\o}rgensen and Smith revisited}\label{se:J-S}

In this section we aim at proving Proposition~\ref{prop:J-S}, which is essentially borrowed from J{\o}rgensen and Smith \cite[Theorem~5-4]{JS90} and its proof there. Following the argument in \cite[Proof of Lemma~5-3]{JS90}, we shall deduce it from Lemma~\ref{lem:J-S}.
\begin{prop}\label{prop:J-S}
  Suppose that $s\in \{0,1\}^\mbbn$ is recurrent and not balanced. Then there exists a factor $w_0$ of $s$ and words $w_1,w_2$ satisfying the following properties.
\begin{itemize}
  \item [(i)] $|w_0|=|w_1|=|w_2|$ and $0<|w_0|_1=|w_1|_1=|w_2|_1<|w_0|$.
  \item [(ii)] Given a balanced pair $\mbba\in\mbfa^2$, there exists $\xi>1$ such that for any finite word $z$, we have:
 \begin{equation}\label{eq:trace larger}
  \max\{\tr[w_1 z]_{\mbba}, \tr [w_2 z]_{\mbba}\} \ge \xi\cdot \tr[w_0 z]_{\mbba}.
\end{equation}
\end{itemize}
\end{prop}

Theorem~\ref{thm:J-S} below generalizes \cite[Theorem~5-4]{JS90} slightly to balanced pairs of matrices. It can be proved by following the original argument of J{\o}rgensen and Smith \cite{JS90} directly.  Since it will be used to prove Lemma~\ref{lem:concave parabolic}, for completeness we provide a proof of it in \S~\ref{se:J-S proof}. For a non-empty finite word $x$, let $x^\infty=xxx\cdots$ denote the periodic infinite word generated by $x$.

\begin{thm}\label{thm:J-S}
Let $\mbba\in\mbfa^2$ be a balanced pair. Let $l,n$ be integers with $n\ge 1$ and $0\le l\le n$, and let $\mclx_{l,n}:=\{x\in \{0,1\}^n : |x|_1=l\}$. Then for $x\in \mclx_{l,n}$, we have: $\tr [x]_{\mbba}\ge \tr [x']_{\mbba}, \forall x'\in \mclx_{l,n}$ iff $x^\infty$ is balanced.
\end{thm}

\subsection{A key lemma}

\begin{lem}\label{lem:J-S}
  Let $(U,V)\in\mbfa^2$ be a balanced pair. Given integers $p,q\ge 1$ and $m\ge 0$, denote 
\[
X=UV, \quad \tX=VU, \quad Y=V^q (UV)^m U^p, \quad \tY=U^p(VU)^mV^q,
\]
and 
\[
W_0= XYX, \quad W_1= \tX Y \tX, \quad W_2= X \tY X.
\]
Then $\eta:=\tr[X (\tY-Y)]>0$, and there exist $t_1,t_2>0$ such that 
\begin{equation}\label{eq:W1}
  W_1 - W_0 = \eta\tX  + t_1(\tX -X),
\end{equation}
\begin{equation}\label{eq:W2}
  W_2 - W_0 = \eta X + t_2(X-\tX).
\end{equation}
\end{lem}

To prove Lemma~\ref{lem:J-S}, let us recall a few facts of $2\times 2$ matrices listed below. 

\begin{lem}\label{lem:Cayley-Hamilton}
Let $X,Y\in \SL(2,\mbbr)$. Then the following hold.
\begin{itemize}
  \item [(i)] $X+X^{-1}=(\tr X)\cdot I$.
  \item [(ii)] $XYX=\tr(XY)\cdot X- Y^{-1}$.
\end{itemize}
\end{lem}
\begin{proof}
  Applying Cayley-Hamilton Theorem to $X$ and $XY$ respectively, the conclusion follows.  
\end{proof}

For $X\in \mbfh$ similar to $\begin{pmatrix} \lambda &  \\ & \lambda^{-1}\end{pmatrix}$, denote $\Gamma_k(X)=\frac{\lambda^k-\lambda^{-k}}{\lambda-\lambda^{-1}}$ for each $k\ge 0$. 
\begin{lem}\label{lem:Chebyshev}
  Let $X,Y\in \mbfh$. Then the following hold.
\begin{itemize}
  \item [(i)] $X^k =\Gamma_k(X)\cdot X -\Gamma_{k-1}(X)\cdot I, \,\forall k\ge 1$.
  \item [(ii)] $X^pY^q-Y^qX^p=\Gamma_p(X)\Gamma_q(Y)\cdot (XY-YX), \,\forall p,q\ge 0$.
\end{itemize}
\end{lem}

\begin{proof}
  (i). It can either be found in \cite{JS90} as Equation (2.4) or be verified directly. 
  
(ii). It follows directly from (i). See also \cite[Lemma~5-2]{JS90}. 
\end{proof}

The following result is due to Southcott \cite[Theorem~6.3]{Sou79}; see also \cite[Proposition~3.4]{BL24}.
\begin{lem}\label{lem:mirror}
  Let $w=i_1i_2\cdots i_{n-1}i_n\in \{0,1\}^n$ and let $\tilde{w}=i_ni_{n-1}\cdots i_2i_1$. Then $\tr [w]_{\mbba}=\tr [\tilde{w}]_{\mbba}$ for any $\mbba\in \mbfg^2$.
\end{lem}

\begin{proof}[Proof of Lemma~\ref{lem:J-S}]
Denote $\alpha_k=\Gamma_k(U),\beta_k=\Gamma_k(V),\gamma_k=\Gamma_k(X)=\Gamma_k(\tX)$ for each $k\ge 0$. We first claim that

\begin{equation}\label{eq:Y diff}
  \tY - Y =\delta (X-\tX), \quad\text{where}\quad \delta:= \gamma_{m+1}\alpha_p\beta_q  - \gamma_m\alpha_{p-1}\beta_{q-1}>0.
\end{equation}
To prove \eqref{eq:Y diff}, first note that by Lemma~\ref{lem:Chebyshev}~(i), $\tX^m=\gamma_{m+1}\cdot I -\gamma_{m}\cdot\tX^{-1}$, and hence
\[
\tY=U^p(\gamma_{m+1}\cdot I -\gamma_{m}\cdot\tX^{-1})V^q = \gamma_{m+1} \cdot U^pV^q - \gamma_m\cdot U^{p-1}V^{q-1}.
\]
Similarly, we have:
\[
Y = \gamma_{m+1}\cdot V^qU^p - \gamma_m\cdot V^{q-1}U^{p-1}.
\]
Also note that by Lemma~\ref{lem:balanced pair}~(i), $X,\tX\in\mbfh$. Then it follows that
\[
\tY - Y = \gamma_{m+1} (U^pV^q- V^qU^p) - \gamma_m(U^{p-1}V^{q-1}-V^{q-1}U^{p-1}) =\delta(X-\tX),
\]
where in the latter equality we have used Lemma~\ref{lem:Chebyshev}~(ii). To see $\delta>0$, we only need to note that $\Gamma_0(\cdot)=0$ and  $\Gamma_k(\cdot)$ is strictly increasing in $k$. 

Due to \eqref{eq:Y diff}, we have:
\[
\eta=\tr[X (\tY-Y)] =\delta\cdot \tr (X^2 -X\tX).
\]
Since $(U,V)\in\mbfa^2$ is a balanced pair, by Lemma~\ref{lem:balanced pair}~(iv), it follows that $\eta>0$.

To proceed, note that by Lemma~\ref{lem:Cayley-Hamilton}~(ii), we have:
\[
W_0=\tr(XY) \cdot X - Y^{-1}, \quad W_1  = \tr(\tX Y)\cdot \tX - Y^{-1}, \quad W_2 = \tr(X\tY)\cdot X - \tY^{-1}.
\]
Also note that by Lemma~\ref{lem:mirror},
\[
\tr Y =\tr \tY, \quad \tr (X\tY) = \tr(Y\tX) = \tr (\tX Y).
\]
It follows that 
\[
W_1 - W_0 = (\tr(X \tY) - \tr(XY) )\cdot \tX + \tr(XY)\cdot (\tX -X).
\]
That is to say, \eqref{eq:W1} holds for $t_1:=\tr(XY)$; moreover, by Lemma~\ref{lem:balanced pair}~(i), $t_1>2$. On the other hand, due to Lemma~\ref{lem:Cayley-Hamilton}~(i) and $\tr Y =\tr \tY$, we have:
\[
Y^{-1} - \tY^{-1} = (\tr Y \cdot I - Y) - (\tr \tY \cdot I - \tY) = \tY - Y.
\]
It follows that 
\[
W_2 - W_0 =  (\tr(X \tY) - \tr(XY) )\cdot X  + (\tY - Y) = \eta X + \delta(X-\tX),
\]
which completes the proof.
\end{proof}

\subsection{Proof of Proposition~\ref{prop:J-S}}

In this subsection we shall deduce Proposition~\ref{prop:J-S} from Lemma~\ref{lem:J-S}, with the help of Lemma~\ref{lem:tree} below. 

Let $s\in\{a,b\}^\mbbn$ be an infinite word over the alphabet $\{a,b\}$. We shall adopt the following convention: if there exist distinct factors $u,v$ of $s$ such that $s$ can be written as $s=w_1w_2\cdots$ for $w_1,w_2,\cdots\in \{u,v\}$, then we say that $s$ is an infinite word over the alphabet $\{u,v\}$. For finite words we adopt a similar convention.

\begin{lem}\label{lem:tree}
  Suppose that $s\in\{0,1\}^\mbbn$ is not balanced. Then there exist distinct factors $u,v$ of $s$ such that the following hold:
\begin{itemize}
  \item [(i)] $s$ is an infinite word over the alphabet $\{u,v\}$;
  \item [(ii)] $uu,vv$ are factors of $s$;
  \item [(iii)] $([u]_\mbba,[v]_\mbba)\in\mbfa^2$ is a balanced pair provided that $\mbba\in\mbfa^2$ is a balanced pair. 
\end{itemize}
\end{lem}

\begin{proof}
The proof is based on the following self-evident observation with an inductive argument.
\begin{clm}
  Let $t$ be an infinite word over the alphabet $\{a,b\}$. Suppose that:
  \begin{itemize}
    \item $bb$ is not a factor of $t$;
    \item there exists a word $w$ such that both $awa$ and $bwb$ are factors of $t$. 
  \end{itemize}
Then we have: 
\begin{itemize}
  \item  $t$ is an infinite word over the alphabet $\{a,ba\}$;
  \item  $w=aw'$, $w'$ is a word over $\{a,ba\}$, and $aw'a,baw'ba$ are factors of $t$. 
\end{itemize}
\end{clm}
Now let us prove the lemma. Since $s$ is not balanced, by Lemma~\ref{lem:not balanced}, there exists a word $w_0$ such that for $u_0=0,v_0=1$, both $u_0w_0u_0$ and $v_0w_0v_0$ are factors of $s$. Now suppose that for some $k\ge 0$, words $u_k,v_k,w_k$ have been defined to satisfy the following induction hypothesis:
\begin{itemize}
  \item $u_k$ and $v_k$ are distinct factors of both $s$ and $w_k$;
  \item $s$ and $w_k$ are words over $\{u_k,v_k\}$;
  \item $u_kw_ku_k$ and $v_kw_kv_k$ are factors of $s$. 
\end{itemize}
Considering $s$ as an infinite word over $\{u_k,v_k\}$, if $v_kv_k$ is not a factor of $s$, then take $a=u_k,b=v_k$; otherwise, if $u_ku_k$ is not a factor of $s$, then take $a=v_k,b=u_k$. In both cases, we can apply the Claim to $t=s,w=w_k$ (as words over $\{a,b\}$), so that for $u_{k+1}=a,v_{k+1}=ba,w_{k+1}=w'$, $u_{k+1},v_{k+1},w_{k+1}$ still satisfy the induction hypothesis above. Since in this process the length of $w_k$ (as a word over $\{0,1\}$) is strictly decreasing in $k$, the induction must be terminated at some step $k\ge 0$. Taking $u=u_k,v=v_k$ for such $k$, assertions (i) and (ii) in the lemma are satisfied; assertion (iii) follows from the construction of $u,v$ and Lemma~\ref{lem:balanced pair}~(ii).
\end{proof}


\begin{proof}[Proof of Proposition~\ref{prop:J-S}]
Let $u,v$ be factors of $s$ given by Lemma~\ref{lem:tree}. Since $s$ is a recurrent infinite word over $\{u,v\}$ and $uu,vv$ are factors of $s$, there exist $p,q\ge 1$ and $m\ge 0$ such that $w_0=uv^{q+1}(uv)^m u^{p+1}v$ is a factor of $s$; moreover, since $s$ is not balanced, both $0$ and $1$ are factors of $w_0$. Take $w_1=vuv^q(uv)^m u^pvu$ and $w_2=uvu^p(vu)^m v^quv$. Then assertion (i) holds. 

To prove (ii), denote $W_i=[w_i]_\mbba$ for $i=0,1,2$ and $U=[u]_\mbba,V=[v]_\mbba$, $X=UV,\tX=VU$. Then by Lemma~\ref{lem:tree}~(iii), $(U,V)\in\mbfa^2$ is a balanced pair, so both \eqref{eq:W1} and \eqref{eq:W2} hold for $\eta,t_1,t_2$ given by Lemma~\ref{lem:J-S}. 

To proceed, we may further assume that $\mbba$ is in a normal form. Then both $X$ and $\tX$ are positive. As a result, there exists $\delta>0$ such that both $X-\delta W_0$ and $\tX-\delta W_0$ are non-negative. Now let us show that \eqref{eq:trace larger} holds for $\xi=1+\eta\delta$. Denote $Z=[z]_{\mbba}$, which is also non-negative. If $\tr (\tX Z)\ge \tr (XZ)$, then by \eqref{eq:W1}, 
\[
\tr (W_1Z) \ge \tr (W_0Z) + \eta\cdot \tr (\tX Z);
\]
combining this with $\tr((\tX-\delta W_0)Z)\ge 0$ yields that $\tr (W_1Z)\ge \xi\cdot\tr (W_0Z)$. Otherwise, a similar argument shows that $\tr (W_2Z)\ge \xi\cdot\tr (W_0Z)$.
\end{proof}

\section{Proof of the main theorem}\label{se:main}

This section is devoted to the proof of Theorem~\ref{thm:main}. We shall fix a (sub-multiplicative) matrix norm $\|\cdot\|$ in this section. In view of Lemma~\ref{lem:Vieta}, we choose to use the {\bf Hilbert-Schmidt norm}; namely, for a $2\times 2$ real matrix $X=\begin{pmatrix} a & b \\ c & d\end{pmatrix}$, $\|X\|$ is defined by 
\[
\|X\|^2=\tr(X^tX)=a^2+b^2+c^2+d^2.
\]
Note that $\tr X\le \sqrt{2}\cdot\|X\|$ holds for any $X$. We shall use this fact occasionally without further explanation.

To deduce Theorem~\ref{thm:main} from Proposition~\ref{prop:J-S}, the starting point of our argument consists of some well-known facts about joint spectral radius summarized in Lemma~\ref{lem:extremal}. The proofs of these facts are usually based on the existence of an {\bf extremal norm}, which is a fundamental tool for the study of joint spectral radius. We refer to \cite{Jun09,BG19} and references therein for more information on extremal norms. In the following let $\supp\mu$ denote the support of a measure $\mu$. 
\begin{lem}\label{lem:extremal}
  Let $(A,B)\in\mbfg^2$ be a balanced pair. Then there exists $C_\mbba>1$ such that the following hold. 
\begin{itemize}
  \item [(i)] $\|[w]_\mbba\|\le C_\mbba\cdot\hrho(\mbba)^{|w|}$ holds for any finite word $w$.
  \item [(ii)] Let $\mu$ be a maximizing measure of $\mbba$. Then for any $s\in\supp\mu$ and any factor $w$ of $s$, $\|[w]_\mbba\|\ge C_\mbba^{-1}\cdot\hrho(\mbba)^{|w|}$. As a result, any invariant measure supported in $\supp\mu$ is also a maximizing measure of $\mbba$.
\end{itemize}
\end{lem}

\begin{proof}
A balanced pair is {\bf irreducible} in the sense of \cite[Page~16]{Jun09}, and hence (i) follows from \cite[Theorem~2.1]{Jun09}. (ii) follows from (i) and \cite[Theorem~2.3~(iv)]{Mor13}; see also \cite[Proposition~6.2]{BG19}.
\end{proof}

Besides Proposition~\ref{prop:J-S} and Lemma~\ref{lem:extremal}, another building block for our proof of Theorem~\ref{thm:main} will be given in Lemma~\ref{lem:concave parabolic}, which generalizes a result of Laskawiec \cite{Las25} (see Lemma~\ref{lem:concave}) slightly. Admitting Lemma~\ref{lem:concave parabolic}, the main idea of our argument is quite simple: if the conclusion of the theorem fails, then we can find a maximizing measure $\mu$ of $\mbba$ and a recurrent $s\in\supp\mu$ which is not balanced. Then applying Proposition~\ref{prop:J-S} to this $s$, for every $n\ge 1$ we can find a factor $t$ of $s$ of the form \eqref{eq:non-balance n} and a corresponding word $t'$ such that \eqref{eq:contradiction norm} holds, which will contradict Lemma~\ref{lem:extremal} and hence complete the proof. To make the logic in the detailed proof more transparent, we shall first prove Theorem~\ref{thm:main} for a co-parallel pair $\mbba\in\mbfh^2$ in \S~\ref{sse:co-para}, and then complete the more technical proof for a general balanced pair in \S~\ref{sse:general}.

\subsection{The co-parallel case}\label{sse:co-para} In this subsection we shall prove Theorem~\ref{thm:main} for a co-parallel pair $\mbba\in\mbfh^2$. To apply Proposition~\ref{prop:J-S}, we need Lemma~\ref{lem:trace-norm} below to control the norm of a matrix in $\mbfh$ by its trace. Recall that for $X=\begin{pmatrix} a & b \\ c & d\end{pmatrix}\in \mbfh$, $s_X$ and $u_X$ denote the attracting and repelling fixed points of $\underline{X}$ respectively.

\begin{lem}\label{lem:Vieta}
  Let $X=\begin{pmatrix} a & b \\ c & d\end{pmatrix}\in \mbfh$ be positive. Then we have:
\begin{itemize}
  \item [(i)] $s_Xu_X=-b/c$;
  \item [(ii)] if $\theta \le b/c\le \theta^{-1}$ for some $\theta\in(0,1)$, then $\tr X\ge \sqrt{\theta}\cdot\|X\|$.
\end{itemize}

\end{lem}

\begin{proof}
(i). It follows from Vieta's Formula. 

(ii). Since $(\tr X)^2-2=a^2+2bc+d^2$ and since $2bc\ge \theta(b^2+c^2)$, the conclusion follows.
\end{proof}

\begin{lem}\label{lem:trace-norm}
Let $\mbba\in\mbfh^2$ be a co-parallel pair. Then there exists $\delta_\mbba>0$ such that $\tr [x]_\mbba \ge \delta_\mbba\cdot\|[x]_\mbba\|$ for any finite word $x$.
\end{lem}

\begin{proof}
Denote $X=[x]_{\mbba}=\begin{pmatrix} a & b \\ c & d\end{pmatrix}$. We may assume that $x\ne\epsilon$ and $\mbba$ is in a normal form, i.e., case (h) in Definition~\ref{defn:balanced pair}. Then by Lemma~\ref{lem:cone}, $X\in\mbfh$, $s_X\in [s_1,s_2]$ and $u_X\in [u_2,u_1]$, so by Lemma~\ref{lem:Vieta}~(i), $-s_1u_1\le b/c\le -s_2u_2$. Thus the conclusion follows from Lemma~\ref{lem:Vieta}~(ii).

\end{proof}

The following statement was proved in \cite[Proof of Theorem~9.8]{Las25}; see also \cite[Proposition~3.4]{MS13} for a closely related statement. Actually Laskawiec \cite[Theorem~9.8]{Las25} also proved the existence of a Sturmian maximizing measure for every co-parallel pair, but in our argument we only need the uniqueness part. Recall that $\mu_\alpha$ denotes the Sturmian measure of slope $\alpha$.

\begin{lem}\label{lem:concave} 
Let $\mbba\in\mbfh^2$ be a co-parallel pair. Then the function $\alpha\mapsto \chi(\mbba,\mu_\alpha)$ is continuous and strictly concave on $[0,1]$. Therefore, there exists at most one $\alpha\in[0,1]$ such that $\mu_\alpha$ is a maximizing measure of $\mbba$. 
\end{lem}

\begin{proof}[Proof of Theorem~\ref{thm:main} for a co-parallel pair in $\mbfh^2$]

Let $\mbba\in\mbfh^2$ be a co-parallel pair. By reduction to absurdity, assume that the conclusion of Theorem~\ref{thm:main} fails for $\mbba$. Then, since $\mbba$ admits some maximizing measure, by Lemma~\ref{lem:concave}, there exists a maximizing measure $\mu$ of $\mbba$ which is not Sturmian. By Lemma~\ref{lem:extremal}~(ii), any invariant measure supported in $\supp\mu$ is also maximizing, so by Lemma~\ref{lem:concave} and ergodic decomposition, we may further assume that $\mu$ is ergodic. Since $\mu$ is ergodic and not Sturmian, by Lemma~\ref{lem:recurrent and balanced} and Lemma~\ref{lem:Sturmian slope}~(i), there exists $s\in\supp\mu$ such that $s$ is recurrent and not balanced. 

Fixing $s$ as above, let $w_0,w_1,w_2$ and  $\xi$ be as given in Proposition~\ref{prop:J-S}. Since $s$ is recurrent and $w_0$ is a factor of $s$, for an arbitrarily large $n$, there exist words $x_1,x_2,\cdots,x_n$ such that 
\begin{equation}\label{eq:non-balance n}
  t:=w_0x_1w_0x_2\cdots w_0x_nw_0
\end{equation}
is a factor of $s$.

\begin{clm}
There exist $i_k\in\{1,2\}, 1\le k\le n$ such that
\begin{equation}\label{eq:contradiction norm}
\|[t']_{\mbba}\| \ge \delta\xi^n\cdot \|[t]_{\mbba}\| \qquad\text{for}\quad t':=w_{i_1}x_1w_{i_2}x_2\cdots w_{i_n}x_nw_0,
\end{equation}
where $\delta>0$ depends only on $\mbba$. 
\end{clm}

\begin{proof}[Proof of Claim]
Due to Lemma~\ref{lem:trace-norm}, it suffices to find $i_k\in\{1,2\}, 1\le k\le n$ such that the following holds:
\begin{equation}\label{eq:contradiction}
  \tr[t']_{\mbba} \ge \xi^n\cdot \tr[t]_{\mbba}.
\end{equation}
To define $i_k,1\le k\le n$ inductively on $k$, for each $0\le k\le n$, we introduce the following notations ($u_0=v_n=\epsilon$):
\[
u_k=w_{i_1}x_1\cdots w_{i_k}x_k, \quad v_k=x_{k+1}w_0x_{k+2} \cdots w_0x_{n}w_0, \quad t_k=u_kw_0v_k.
\]
Now given $0\le k <n$, assume that $i_j$ has been chosen for $1\le j\le k$. Then substituting $z=v_ku_k$ into \eqref{eq:trace larger}, we can find $i_{k+1}\in \{1,2\}$ such that
\[
\tr [u_kw_{i_{k+1}}v_k]_{\mbba} = \tr [w_{i_{k+1}}v_ku_k]_{\mbba} \ge \xi\cdot\tr[w_0v_ku_k]_{\mbba} = \xi \cdot\tr[t_k]_{\mbba}.  
\]
Also note that by definition, $t_{k+1}=u_{k+1}w_0v_{k+1}=u_kw_{i_{k+1}}v_k$. Then by induction, $i_1,\cdots,i_n$ have been defined and
\[
\tr[t']_{\mbba} = \tr[t_n]_{\mbba} \ge \xi\cdot \tr[t_{n-1}]_{\mbba} \ge \cdots \ge \xi^n\cdot \tr[t_0]_{\mbba} = \xi^n\cdot \tr[t]_{\mbba},
\]
which completes the proof of \eqref{eq:contradiction}.
\end{proof}
Now combining \eqref{eq:contradiction norm} with Lemma~\ref{lem:extremal} and $|t|=|t'|$, we obtain that 
\[
C_\mbba\cdot\hrho(\mbba)^{|t|} \ge \|[t']_{\mbba}\| \ge \delta\xi^n\cdot \|[t]_{\mbba}\| \ge C_\mbba^{-1}\delta \xi^n\cdot\hrho(\mbba)^{|t|}. 
\]
Since $n$ can be arbitrarily large, this leads to a contradiction.
\end{proof}

\subsection{The general situation}\label{sse:general} In this subsection we complete the proof of Theorem~\ref{thm:main} for an arbitrary balanced pair $\mbba\in\mbfg^2$. For preparation we need to generalize Lemma~\ref{lem:trace-norm} and Lemma~\ref{lem:concave} to Lemma~\ref{lem:trace-norm parabolic} and Lemma~\ref{lem:concave parabolic} respectively. 

\begin{lem}\label{lem:trace-norm parabolic}
Let $\mbba\in\mbfa^2$ be a balanced pair. Given $N\ge 2$, let 
\[
\mclh_N:=\{x\in \cup_{n=N}^\infty\{0,1\}^n: x\ne 0^Ny,1^Ny, y0^N, y1^N, \forall y\in\{0,1\}^*\}.
\]
Then there exists $\delta_N>0$ such that $\tr [x]_\mbba \ge \delta_N\cdot\|[x]_\mbba\|$ holds for any $x\in \mclh_N$.
\end{lem}

\begin{proof}
We may assume that $\mbba$ is in a normal form. For each $x\in\mclh_N$, denote $X=[x]_{\mbba}\in\mbfh$. Denote $\mclx_N=\{x\in\mclh_N:|x|=N\}$. By Lemma~\ref{lem:cone}~(i)(iv), there exist real numbers $u_2<u_1<0<s_1<s_2$ such that 
\[
\bigcup_{x\in \mclh_N} \underline{X}(\overline{I_{\mbba}^+}) = \bigcup_{x\in \mclx_N} \underline{X}(\overline{I_{\mbba}^+}) \subset [s_1,s_2]\,,\, \bigcup_{x\in \mclh_N} \underline{X}^{-1}(\overline{I_{\mbba}^-}) = \bigcup_{x\in \mclx_N} \underline{X}^{-1}(\overline{I_{\mbba}^-}) \subset [u_2,u_1].
\]
Then both $s_X\in [s_1,s_2]$ and $u_X\in [u_2,u_1]$ hold for any $x\in\mclh_N$, and hence the conclusion follows from Lemma~\ref{lem:Vieta}.
\end{proof}

\begin{lem}\label{lem:concave parabolic}
  Let $\mbba\in \mbfg^2$ be a balanced pair. Then the function $\alpha\mapsto\chi(\mbba,\mu_\alpha)$ is continuous and strictly concave on $[0,1]$. Therefore, there exists at most one $\alpha\in[0,1]$ such that $\mu_\alpha$ is a maximizing measure of $\mbba$.
\end{lem}

\begin{proof}
Note that if $\mbba=(aA,bB)$ for some $(A,B)\in\mbfa^2$ and $a,b\in\mbbr\setminus\{0\}$, then by \eqref{eq:LyaExp} and Lemma~\ref{lem:Sturmian slope}~(ii),
\[
\chi(\mbba,\mu_\alpha)= \chi((A,B),\mu_\alpha) + (1-\alpha)\log |a| + \alpha \log |b|, \quad\forall\alpha\in[0,1].
\]
Therefore, we may assume that $\mbba\in \mbfa^2$.  We may further assume that $\mbba$ is in a normal form, so that $I_\mbba^+$ introduced in \eqref{eq:cone} is well-defined. Define $f(\alpha)=\chi(\mbba,\mu_\alpha), \alpha\in[0,1]$. Given $N\ge 3$, let us first show that $f$ is (continuous and) strictly concave on $J_N=(\frac{1}{N},1-\frac{1}{N})$. To this end, we shall follow the argument of Laskawiec in the second paragraph of \cite[Proof~of~Theorem~9.8]{Las25}, which consists of two steps: firstly, prove that $f$ is strictly mid-point concave on $J_N\cap \mbbq$; secondly, prove that $f$ is continuous on $J_N$. 

For the first step, it suffices to show that, under our settings, for $t_1,t_2\in J_N\cap \mbbq$ with $t_1<t_2$, the corresponding words $w_1,w_2,w_3$ and associated matrices $\Pi_i=[w_i]_\mbba, i=1,2,3$ defined in \cite[Proof~of~Theorem~9.8]{Las25} still satisfy 
\[
\rho(\Pi_3)\ge \rho(\Pi_1\Pi_2)>\rho(\Pi_1)\cdot\rho(\Pi_2).
\]
Since $|w_1w_2|=|w_3|,|w_1w_2|_1=|w_3|_1$ and since $w_3^\infty$ is balanced, by Theorem~\ref{thm:J-S}, $\rho(\Pi_3)\ge \rho(\Pi_1\Pi_2)$ holds. On the other hand, by following the proofs of the ``furthermore" part of \cite[Lemma~9.4]{Las25} and \cite[Corollary~9.5]{Las25}, it can be directly checked that $(\Pi_1,\Pi_2)$ is a balanced pair; moreover, by Lemma~\ref{lem:balanced pair}~(i), $\Pi_1,\Pi_2\in\mbfh$. Therefore,  $(\Pi_1,\Pi_2)$ is a co-parallel pair, so by Lemma~\ref{lem:concave}, $\rho(\Pi_1\Pi_2)>\rho(\Pi_1)\cdot\rho(\Pi_2)$. 

For the second step, let 
\begin{equation}\label{eq:sft}
  \Sigma_N=\{s\in \{0,1\}^\mbbn : \text{neither $0^N$ nor $1^N$ is a factor of $s$}\},
\end{equation}
and let $\mclm_N=\{\mu\in\mclm:\supp\mu\subset \Sigma_N\}$. By Lemma~\ref{lem:Sturmian slope}~(ii), $\supp\mu_\alpha\subset \Sigma_N$ holds for every $\alpha\in J_N$. Therefore, to prove that $f$ is continuous on $J_N$, it suffices to replace the role of \cite[Lemma~9.6]{Las25} there with the following claim.

\begin{clm}
 The function $\mu\mapsto\chi(\mbba,\mu),\mu\in\mclm$ is continuous on $\mclm_N$.
\end{clm}

\begin{proof}[Proof of Claim]
Consider 
\[
\Sigma_N'=\{\underline{i}\in \{0,1\}^\mbbz : \text{neither $0^N$ nor $1^N$ is a factor of $\underline{i}$}\}
\]
as the natural extension of $\Sigma_N$ and let $\mclm_N'$ be the collection shift-invariant Borel probability measures for the sub-shift $\Sigma_N'$. The Lyapunov exponent $\chi(\mbba,\nu)$ for $\nu\in \mclm_N'$ can be similarly defined, and it suffices to prove that $\nu\mapsto\chi(\mbba,\nu)$ is continuous on $\mclm_N'$.

Note that $\Sigma_N'$ is a sub-shift of finite type. Also note that by Lemma~\ref{lem:cone}~(iv), if $w$ is a word of length $N$ with $w\ne 0^N,1^N$, then for $W=[w]_{\mbba}$, $\underline{W}(\overline{I_\mbba^+})\subset I_\mbba^+$. Then by \cite[Theorem~2.3]{ABV10}, it can be seen that $\mbba$ is uniformly hyperbolic with respect to $\Sigma_N'$. To proceed, we shall follow the proof of Equation~(1.7) in \cite[\S~2.1]{BM16}. Fixing a norm $|\cdot|$ on $\mbbr^2$, thanks to the uniform hyperbolicity, it is not difficult to show that there exists a continuous map 
\[
e^+:\Sigma_N'\to \{(x,y)^t\in\mbbr^2 : y>0,x/y\in I_{\mbba}^+\}
\] 
such that for every $\underline{i}=(i_n)_{n\in\mbbz}\in\Sigma_N'$, we have (denote $\mbba=(A_0,A_1)$ below): 
\begin{itemize}
  \item $\frac{A_{i_0}e^+(\underline{i})}{|A_{i_0}e^+(\underline{i})|} = \frac{e^+((i_{n+1})_{n\in\mbbz})}{|e^+((i_{n+1})_{n\in\mbbz})|}$;
  \item $\lim_{n\to\infty}\frac{1}{n}\log|A_{i_n}\cdots A_{i_1}A_{i_0}e^+(\underline{i})|=\lim_{n\to\infty}\frac{1}{n}\log\|A_{i_n}\cdots A_{i_1}A_{i_0}\|$, provided that the limit on the right hand side exists.
\end{itemize} 
Indeed, $e^+(\underline{i})$ can be taken as $(x(\underline{i}),1)^t$, where $x(\underline{i})$ denotes the unique point in $\cap_{n=1}^\infty \underline{A_{-i_1}\cdots A_{-i_n}}(I_\mbba^+)$. Then it follows that $\chi(\mbba,\nu)=\int_{\Sigma_N'}\log|A_{i_0}e^+(\underline{i})|\dif\nu(\underline{i})$, which implies that $\chi(\mbba,\nu)$ is continuous in $\nu$.
\end{proof}

Since $N$ can be arbitrarily large, we have proved that $f$ is strictly concave on $(0,1)$. On the other hand, it is easy to see that as $n\to \infty$, $f(\frac{1}{n+1})=\frac{1}{n+1}\log\rho(A^nB)\to \log\rho(A) = f(0)$ and $f(\frac{n}{n+1})=\frac{1}{n+1}\log\rho(AB^n)\to \log\rho(B) = f(1)$. Combining this with $f$ being concave on $(0,1)$, we conclude that $f$ is continuous at $0,1$, which completes the proof.
\end{proof}

\begin{proof}[Proof of Theorem~\ref{thm:main}]
Let us follow the proof in the co-parallel case, arguing by reduction to absurdity. Assuming that the conclusion of Theorem~\ref{thm:main} fails for $\mbba$, then thanks to Lemma~\ref{lem:concave parabolic}, we can still find an ergodic maximizing measure $\mu$ of $\mbba$ which is not Sturmian, and hence there exists $s\in \supp\mu$ which is recurrent and not balanced. Then for every $n$, there exists a factor $t$ of $s$ of the form \eqref{eq:non-balance n}, and to get a contradiction as in the co-parallel case, it suffices to show that \eqref{eq:contradiction norm} still holds for some $\delta>0$ dependent only on $\mbba$ and $|w_0|$. Due to Proposition~\ref{prop:J-S}~(i), $|t|=|t'|$ and $|t|_1=|t'|_1$, so we may assume that $\mbba\in \mbfa^2$. Moreover, $w_0\ne 0^{|w_0|}, 1^{|w_0|}$, so by Lemma~\ref{lem:trace-norm parabolic}, the problem is again reduced to establishing \eqref{eq:contradiction}, which completes the proof.
\end{proof}

\section{Further discussions}\label{se:FD}

In this section we shall discuss some relation between our result and previous results in the literature, and also give some applications of Theorem~\ref{thm:main}. The main results in this section are Proposition~\ref{prop:balanced} and Proposition~\ref{prop:family}. For preparation let us begin with Proposition~\ref{prop:slope cont}.

From now on denote 
\[
\mbfb:=\{\mbba\in \mbfg^2: \mbba \text{ is a balanced pair}\}.
\] 
For every $\mbba\in\mbfb$, by Theorem~\ref{thm:main}, there exists a unique $\tau(\mbba)\in [0,1]$ such that $\mu_{\tau(\mbba)}$ is the unique maximizing measure of $\mbba$. This defines a function $\tau:\mbfb\to [0,1]$.

\begin{prop}\label{prop:slope cont}
Identifying $\mbfb$ as a subset of $\mbbr^8$, then the function $\tau:\mbfb\to[0,1]$ is continuous. 
\end{prop}

\begin{proof}
Thanks to Lemma~\ref{lem:concave parabolic}, we can apply Lemma~\ref{lem:implicit} below to the function $f(\mbba,\alpha)=\chi(\mbba,\mu_\alpha),(\mbba,\alpha)\in\mbfb\times[0,1]$ by taking $D=[0,1]\cap\mbbq$. Then the conclusion follows.
\end{proof}

\begin{lem}\label{lem:implicit}
  Let $X$ be a topological space and suppose that $f:X\times[0,1]\to \mbbr$ satisfies the following properties.
\begin{itemize}
  \item For every $x\in X$, $\alpha\mapsto f(x,\alpha)$ is continuous and strictly concave on $[0,1]$, so this function attains its maximum at a unique point $\tau(x)\in[0,1]$.
  \item There exists a dense subset $D$ of $[0,1]$ such that for every $\alpha\in D$, $x\mapsto f(x,\alpha)$ is continuous on $X$.
\end{itemize}
Then the function $\tau:X\to [0,1]$ is continuous.
\end{lem}

\begin{proof}
  Given $x_0\in X$ and $\ve>0$, it suffices to show that there exists an open neighborhood $U$ of $x_0$ such that $|\tau(x)-\tau(x_0)|<\ve$ for any $x\in U$. Let us only prove the statement under the additional assumption $\tau(x_0)\in (0,1)$; the special cases $\tau(x_0)=0,1$ are similar and omitted.
  
Since $\tau(x_0)\in (0,1)$ and $D$ is dense in $[0,1]$, there exist $\alpha_1,\alpha_2,\beta_1,\beta_2\in D$ with $\tau(x_0)-\ve <\alpha_1<\alpha_2<\tau(x_0)<\beta_1<\beta_2<\tau(x_0)+\ve$. Since the function $\alpha\mapsto f(x_0,\alpha)$ is strictly concave and since it attains its maximum at $\tau(x_0)$, we have $f(x_0,\alpha_1)<f(x_0,\alpha_2)$ and $f(x_0,\beta_1)>f(x_0,\beta_2)$. Then by continuity, there exists an open neighborhood $U$ of $x_0$ such that both $f(x,\alpha_1)<f(x,\alpha_2)$ and $f(x,\beta_1)>f(x,\beta_2)$ hold for any $x\in U$. For such $x$, by the concavity of $\alpha\mapsto f(x,\alpha)$, we have $\alpha_1<\tau(x)<\beta_2$.
\end{proof}

All the finiteness counter-examples in \cite{BM02,BTV03,Koz05,HMST11,MS13,JP18,Ore18} are constructed in the following way: for certain pair of non-negative matrices $(A,B)\in \mbfg^2$, prove that there exist finiteness counter-examples in the one-parameter family $(A,tB),t>0$. The proposition below shows that such $(A,B)$ must be a balanced pair.

\begin{prop}\label{prop:balanced}
Let $(A,B)\in\mbfg^2$ and suppose that there exists $T\in\GL(2,\mbbr)$ such that both $T^{-1}AT$ and $T^{-1}BT$ are non-negative. Then $(A,B)$ is a balanced pair iff there exists $t>0$ such that $(A,tB)$ is a finiteness counter-example.
\end{prop}

\begin{proof}
Let us prove the ``only if" part first. For every $t>0$, since $(A,tB)$ is a balanced pair, by Theorem~\ref{thm:main}, there exists a unique $\alpha(t)\in [0,1]$ such that $\mu_{\alpha(t)}$ is the unique maximizing measure of $(A,tB)$. By Proposition~\ref{prop:slope cont}, the function $t\mapsto\alpha(t)$ is continuous on $(0,+\infty)$. By the definition of $\alpha(t)$, it is easy to see that $\lim_{t\to 0^+}\alpha(t)=0$ and $\lim_{t\to +\infty}\alpha(t)=1$. Then there exists (uncountably many) $t>0$ with $\alpha(t)\notin\mbbq$, and $(A,tB)$ is a finiteness counter-example for such $t$.  

For the ``if" part, we may assume that $A,B\in\mbfg$ are both non-negative. Since for $r,s,t>0$, $(A,tB)$ is a balanced pair iff $(rA,sB)$ is a balanced pair, it suffices to prove the following statement.

\begin{clm}
  Let $A,B\in\SL(2,\mbbr)\setminus\{I\}$ be non-negative matrices. If $\mbba=(A,B)$ does not admit a spectrum maximizing product, then $\mbba$ is a balanced pair.
\end{clm} 
It remains to prove the claim. To begin with, note that if $C\in\SL(2,\mbbr)\setminus\{I\}$ is non-negative, then 
$C\in\mbfa$, and hence either $C$ is in the form of \eqref{eq:hyperbolic}, or $C\in\{P_x,P_x^t\}$ for some $x>0$. If $A$ and $B$ share a common eigenvector, it is well-known (and easy to see) that either $A$ or $B$ is a spectrum maximizing product of $\mbba$. Then we may assume that $A$ and $B$ share no common eigenvector, so there are three cases as follows.

\begin{itemize}
  \item $\{A,B\}=\{H_{s_1,u_1}^{\lambda_1},H_{s_2,u_2}^{\lambda_2}\}$ with $s_1\ne s_2$ and $u_1\ne u_2$.
  \item $\{A,B\}=\{H_{s,u}^\lambda,P_x\}$ or $\{H_{s,u}^\lambda,P_x^t\}$.
  \item $\{A,B\}=\{P_x,P_y^t\}$. 
\end{itemize}
Therefore, $\mbba$ is either a balanced pair or a crossing pair (recall its definition ahead of Lemma~\ref{lem:balanced pair}). According to \cite[Theorem~3.4]{PS21} or \cite[Corollary~5.4]{Las25},  for a crossing pair $(A,B)$, either $A$ or $B$ is a spectrum maximizing product. The proof is done.
\end{proof}

\begin{rmk}
 From Proposition~\ref{prop:balanced} we know that a result of Morris and Sidorov \cite[Theorem~2.3]{MS13} can be applied to any balanced pair. In particular, given $(A,B)\in\mbfb$, for every $t>0$, letting $\mu_{\alpha(t)}$ be the unique maximizing measure of $(A,tB)$, then we have: the set $\{t>0:\alpha(t)\notin \mbbq\}$ is of Hausdorff dimension zero. As an immediate corollary of this fact and Fubini's theorem, identifying $\mbfb$ as a Borel subset of $\mbbr^8$ (the interior of $\mbfb$ consists of co-parallel pairs), we have: for Lebesgue almost every $\mbba\in\mbfb$, $\mbba$ admits a unique maximizing measure and this measure is supported on a periodic orbit. 
\end{rmk}

To prove Proposition~\ref{prop:family}, we need a result of Panti and Sclosa \cite[Theorem~3.6]{PS21} as follows. Let us give an alternative proof of it as an application of Theorem~\ref{thm:main}.

\begin{lem}\label{lem:P-S}
  Let $(A,B)\in\mbfa^2$ be a balanced pair with $\tr A=\tr B$. Then its unique maximizing measure is $\mu_{1/2}$.
\end{lem}

\begin{proof}
For $A,B$ given in the lemma, it is not difficult to prove the following statement directly: there exists $T\in\GL(2,\mbbr)$ such that $T^{-1}AT=B$ and $T^{-1}BT=A$. Alternatively, this statement follows easily from a more general result of Friedland \cite[Theorem~2.1,~Theorem~2.9]{Fri83}; see also \cite[Lemma~2.2]{Koz25}.

Now suppose that $\mu_\alpha$ is the unique maximizing measure of $(A,B)$ for some $\alpha\in[0,1]$. Then $\mu_{1-\alpha}$ is the unique maximizing measure of $(B,A)$. However, from $(T^{-1}AT,T^{-1}BT)=(B,A)$ we know that $\mu_\alpha$ is also a maximizing measure of $(B,A)$, so $\alpha=1/2$. 
\end{proof}

\begin{prop}\label{prop:family}
Fix $u'<u<0<s<s'$ and $\lambda'>1$. Then there exist uncountably many  $\lambda\in[\lambda',+\infty)$ such that the co-parallel pair $\mbba_{\lambda}=(H_{s,u}^{\lambda}, H_{s',u'}^{\lambda'})$ is a finiteness counter-example.
\end{prop}

\begin{proof}
  For each $\lambda\in [\lambda',+\infty)$, let $\mu_{f(\lambda)}$ be the unique maximizing measure of $\mbba_{\lambda}$. By Proposition~\ref{prop:slope cont}, this defines a continuous function $f:[\lambda',+\infty)\to [0,1]$. By Lemma~\ref{lem:P-S}, $f(\lambda')=1/2$. On the other hand, $\chi(\mbba_\lambda,\mu_0)=\log\lambda$, and it is easy to see that there exists $C>0$ such that
\[
\chi(\mbba_\lambda,\mu_\alpha)\le (1-\alpha)\log\lambda + C, \quad \forall \alpha\in (0,1],\forall\lambda\in[\lambda',+\infty).
\]
Then $\lim_{\lambda\to +\infty}f(\lambda)=0$, and the conclusion follows.
\end{proof}

\section*{Acknowledgements} 

The author would like to thank Bing Gao for valuable comments. This work is partially supported by the Taishan Scholar Project of Shandong Province (No. tsqnz20230614) and the National Natural Science Foundation of China (12571201).

\appendix

\section{Proof of Theorem~\ref{thm:J-S}}\label{se:J-S proof}


In this appendix we prove Theorem~\ref{thm:J-S} by modifying the proof of Proposition~\ref{prop:J-S}. To begin with we need Lemma~\ref{lem:cyclic} below to replace the role of Lemma~\ref{lem:tree} in the proof of Proposition~\ref{prop:J-S}.

\begin{lem}\label{lem:cyclic}
Let $x\in \{0,1\}^*$ and suppose that $x^\infty$ is not balanced. Then there exists a factor $\tx$ of $x^\infty$ with $|\tx|=|x|$ and distinct factors $u,v$ of $\tx$ such that the following hold:
\begin{itemize}
  \item [(i)] $\tx$ is a word over the alphabet $\{u,v\}$;
  \item [(ii)] $uu,vv$ are factors of $\tx^\infty$;
  \item [(iii)] $([u]_\mbba,[v]_\mbba)\in\mbfa^2$ is a balanced pair provided that $\mbba\in\mbfa^2$ is a balanced pair. 
\end{itemize}
\end{lem}

\begin{proof}
Similar to Lemma~\ref{lem:tree}, the proof is based on the following self-evident observation with an inductive argument.
\begin{clm}
  Let $z$ be a finite word over the alphabet $\{a,b\}$. Suppose that:
  \begin{itemize}
    \item $bb$ is not a factor of $z^\infty$;
    \item there exists a word $w$ such that both $awa$ and $bwb$ are factors of $z$. 
  \end{itemize}
Denote $z=cy$ with $c\in\{a,b\}$. If $c=a$, let $z'=ya$; otherwise, let $z'=z$. Denote $w=aw'$. Then we have: 
\begin{itemize}
  \item  $z',w'$ are words over the alphabet $\{a,ba\}$;
  \item  $aw'a,baw'ba$ are factors of $z'$. 
\end{itemize}
\end{clm}
Now let us prove the lemma. Since $x^\infty$ is not balanced, it is known that there exists a factor $x_0$ of $x^\infty$ with $|x_0|=|x|$ such that $x_0$ is not balanced; see, for example, ``(i) implying (iii)" in \cite[Lemma~4.7]{HMST11}. Then by Lemma~\ref{lem:not balanced}, there exists a word $w_0$ such that for $u_0=0,v_0=1$, both $u_0w_0u_0$ and $v_0w_0v_0$ are factors of $x_0$. Now suppose that for some $k\ge 0$, words $x_k,u_k,v_k,w_k$ have been defined to satisfy the following induction hypothesis:
\begin{itemize}
  \item $x_k$ is a factor of $x^\infty$ with $|x_k|=|x|$ (as words over $\{0,1\}$);
  \item $u_k$ and $v_k$ are distinct factors of both $x_k$ and $w_k$;
  \item $x_k$ and $w_k$ are words over $\{u_k,v_k\}$;
  \item $u_kw_ku_k$ and $v_kw_kv_k$ are factors of $x_k$. 
\end{itemize}
Considering $x_k^\infty$ as an infinite word over $\{u_k,v_k\}$, if $v_kv_k$ is not a factor of $x_k^\infty$, then take $a=u_k,b=v_k$; otherwise, if $u_ku_k$ is not a factor of $x_k^\infty$, then take $a=v_k,b=u_k$. In both cases, we can apply the Claim to $z=x_k,w=w_k$ (as words over $\{a,b\}$), so that for $x_{k+1}=z',u_{k+1}=a,v_{k+1}=ba,w_{k+1}=w'$, $x_{k+1},u_{k+1},v_{k+1},w_{k+1}$ still satisfy the induction hypothesis above. Since in this process the length of $w_k$ (as a word over $\{0,1\}$) is strictly decreasing in $k$, the induction must be terminated at some step $k\ge 0$. Taking $\tx=x_k,u=u_k,v=v_k$ for such $k$, assertions (i) and (ii) in the lemma are satisfied; assertion (iii) follows from the construction of $u,v$ and Lemma~\ref{lem:balanced pair}~(ii).
\end{proof}

\begin{proof}[Proof of Theorem~\ref{thm:J-S}]
Let $x\in\mclx_{l,n}$ and suppose that $x^\infty$ is not balanced. It suffices to find $x'\in \mclx_{l,n}$ such that $\tr [x']_{\mbba}>\tr[x]_{\mbba}$. To this end, let $\tx$ and $u,v$ be as given in Lemma~\ref{lem:cyclic}. Due to Lemma~\ref{lem:cyclic}~(i)(ii), there exists a factor $y$ of $\tx^\infty$ with $|y|=|\tx|$ and $p,q\ge 1,m\ge 0$ such that:
\begin{itemize}
  \item either $y=w_0z$ for $w_0=uv^{q+1}(uv)^m u^{p+1}v$ and some word $z$ over $\{u,v\}$; 
  \item or $y=w_0z$ for $w_0=vu^{p+1}(vu)^m v^{q+1}u$ and some word $z$ over $\{u,v\}$; 
  \item or $y=u^{p+1}v^{q+1}$.
\end{itemize}
Note that $y\in \mclx_{l,n}$ and $\tr [y]_{\mbba}=\tr[x]_{\mbba}$. In the first case, we can choose $w_1,w_2$ as in the proof of Proposition~\ref{prop:J-S}, so that \eqref{eq:trace larger} holds for some $\xi>1$. Then $w_1z,w_2z\in \mclx_{l,n}$ and the desired inequality holds for some $x'\in\{w_1z,w_2z\}$. The second case can be dealt with in the same way as the first one. In the last case, let us take $x'=u^pv^quv \in \mclx_{l,n}$. Then for $U=[u]_\mbba,V=[v]_\mbba$, $(U,V)\in\mbfa^2$ is a balanced pair, and we have:
\[
\tr [x']_{\mbba} - \tr[x]_{\mbba} = \tr [(U^pV^q-V^qU^p)UV]=\Gamma_p(U)\Gamma_q(V)\cdot\tr((UV)^2-U^2V^2) >0,
\] 
where the latter ``$=$" is due to Lemma~\ref{lem:Chebyshev}~(ii), while the ``$>$" is due to $p,q\ge 1$ and Lemma~\ref{lem:balanced pair}~(iv). The proof is done. 
\end{proof}


\bibliographystyle{plain}             
\bibliography{refer}

\end{document}